\documentclass[12pt]{amsart}
\usepackage{amsmath, amssymb, amscd, a4wide}

\usepackage[colorlinks=true,linkcolor=blue,citecolor=blue,pdfpagelabels=false]{hyperref}
\numberwithin{equation}{section}

\theoremstyle{plain}
\newtheorem{thm}{Theorem}[section]
\newtheorem{lem}[thm]{Lemma}

\newcommand{\thmref}[1]{Theorem~\ref{#1}}
\newcommand{\lemref}[1]{Lemma~\ref{#1}}

\theoremstyle{definition}
\newtheorem{rmk}[thm]{Remark}

\newcommand{\rmkref}[1]{Remark~\ref{#1}}

\newcommand{\mbb}{\mathbb}
\newcommand{\x}{\textbf}
\newcommand{\mf}{\mathbf}
\newcommand{\q}{\quad}

\newcommand{\mc}{\mathcal}

\newcommand{\mrm}{\mathrm}

\begin{document}

\title[$L^\infty$ norms of holomorphic modular forms]{$L^\infty$ norms of holomorphic modular forms in the case of compact quotient}

\author{Soumya Das} 
\address{Department of Mathematics\\
Indian Institute of Science\\
Bangalore -- 560012, India.}
\email{somu@math.iisc.ernet.in}

\author{Jyoti Sengupta}
\address{School of Mathematics\\
Tata Institute of Fundamental Research\\
Homi Bhabha Road\\
Mumbai -- 400005, India.}
\email{sengupta@math.tifr.res.in}


\begin{abstract}
We prove a sub-convex estimate for the sup-norm of $L^2$-normalized holomorphic modular forms of weight $k$ on the upper half plane, with respect to the unit group of a quaternion division algebra over $\mf Q$. More precisely we show that when the $L^2$ norm of an eigenfunction $f$ is one,\[ \| f \|_\infty \ll k^{\frac{1}{2} - \frac{2}{131} +\varepsilon}\  \] for any $\varepsilon>0$ and for all $k$ sufficiently large.
\end{abstract}
\maketitle 

\section{Introduction}
The supremum norm of cusp forms has been a topic of considerable interest in the recent past. Let us first look at the case of holomorphic cusp forms of weight $k$ for the full modular group, $SL(2,\mf Z)$. Let $f$ be such a form. We further assume that $f$ is a  Peterson normalised eigenfunction of all the Hecke operators. Then the $L^\infty$ norm of $f$ is by definition the supremum of the bounded $SL(2,\mf Z)$-invariant function $y^{k/2} |f(z)| $: \[ \| f \|_\infty = \underset{z \in \mf H} \sup \ |y^{k/2} f(z)|, \] 
where $z=x+iy$ the Poincare upper half-plane ${\mf H}$. In \cite{xia},  
H. Xia proved that 

\begin{align} \label{xia}
 k^{1/4 - \epsilon} \ll \| f \|_\infty\ll k^{\frac{1}{4}+\epsilon} \ 
\text{ for all } \  \epsilon > 0.
\end{align}
Note that the convexity or `trivial' bound in this case is $\| f \|_\infty \ll k^{\frac{1}{2}+\epsilon} $.

In the case of Maass forms of weight zero, Iwaniec and Sarnak showed in  an important paper (\cite{is}) that 
\[ \| f \|_\infty\ \ll \lambda^{\frac{5}{24}+\epsilon} \ \text{ for all } \ 
\epsilon  > 0, \] 
where 
$\lambda$ is the eigenvalue of $f$ for the hyperbolic Laplacian.  Here $f$ has $L^2$ (or Petersson) norm one. Iwaniec and Sarnak also investigate the supermum norm of eigenfunctions on a compact arithmetic surface.  Such a surface is of the form $\Gamma \backslash \mf H$ where $\Gamma$ is a cocompact arithmetic subgroup of $SL(2,\mf R)$ arising from quaternion division algebra over $\mf Q$. They considered the supermum norm of eignefunctions of the Laplacian on $\Gamma \backslash \mf H$ under the assumption that the eigenfunction in question is also a simultaneous eigenfunction of the Hecke operators $T(n), (n,q)=1$. Here $q$ is a positive integer depending on the maximal order $R$ in the quaternion division algebra which gives rise to $\Gamma = R(1)$, the group of units of $R$. The result they prove is the same as in the Maass case, i.e., if $f$ is any such eigenfunction with $L^2$ norm one then \[ \| f\| \ll
\lambda^{\frac{5}{24} +\epsilon} \ \text{ for all } \ \epsilon >0. \] The convexity bound here is $\| f\|_\infty  \ll \lambda^{\frac{1}{4} +\epsilon}$.

In this note we place ourselves in the same setting as Iwaniec-Sarnak i.e., we consider a cocompact arithmetic subgroup $\Gamma$ as above. However the functions we consider are holomorphic modular forms for $\Gamma$ of weight $k$ where $k$ is a positive even integer. Recall that the $L^\infty$ norm is the supermum of the $\Gamma$ invariant function $y^{k/2} |f(z)|$ if $f$ has weight $k$.  In this situation we prove the following result.

\begin{thm} \label{mainthm}
Let $\Gamma$ be as above and $f$ a holomorphic modular
form for $\Gamma$ of weight $k$. Assume that $f$ is a simultaneous 
eigenfunction of all the Hecke operators. Assume that $f$ has Petersson norm one. Then for all $\varepsilon>0$ there exist an absolute constant $k_0>0$ such that for all $k > k_0$,
\[ \| f \|_\infty \ll k^{\frac{1}{2} - \frac{2}{131} +\varepsilon}\  .\]
The implied constant depends on $\varepsilon$ and the group $\Gamma$ but not on $f$.
\end{thm}

Note that the convexity bound in this case is $ \| f \|_\infty \ll k^{\frac{1}{2}}$, which is sharp in some cases if the Hecke assumption is removed; see \rmkref{nohecke} in section~\ref{convex}.

In Xia's argument while obtaining \eqref{xia}, with both the upper and lower bounds, essential use is made of the presence of a cusp and the Fourier expansion of $f$ in the noncompact case. In fact, this allows
him to use Deligne's sharp bound for the Fourier coefficients for
the upper bound while taking the point $z$ very high up in the cusp allows for the lower bound. In the setting of our paper, there are no cusps and both of these tools are lost. 

Our approach consists in employing the Bergman kernel for the compact
quotient $\Gamma \backslash \mf H$. We embedd $f$ in a orthonormal basis $\{f_j\}$ of the space of modular forms of weight $k$, each $f_j$ being a simultaneous Hecke eigenform. Recall that the Bergman kernel $h_k(z,w)$ is proportional to $\sum_j f_j(z) \overline{f_j(w)}$. We apply the Hecke operator $T(n)$ in the $w$-variable and then estimate the resulting function. We first derive a reasonable estimate for $h_k(z,z)$ using some results of Cogdell and Luo (\cite{cl}), which is presented in \eqref{hnfinal}. We next implement the amplification technique of Iwaniec and Sarnak to highlight the contribution of $f$ and obtain the result.

In particular one does not have a direct $k^{1/4}$ upper bound as in \cite{xia}, while it is possible that even an upper bound $k^{\varepsilon}$ might hold in \thmref{mainthm}.

\section*{Acknowledgements} It is a great pleasure for the authors to thank Prof. Peter Sarnak for his thoughtful comments on the paper. We also thank the School of Mathematics TIFR, Mumbai and the Department of Mathematics, IISc., Bangalore where parts of this work was carried out, for providing excellent working atmosphere. The first author was partly financially supported by the DST-INSPIRE Scheme IFA 12-MA-13.

\section{Notation and setup}
\subsection{Quaternion algebras and orders}
Let $A = \left( \frac{a,b}{\mf Q} \right)$ be a quaternion division algebra over $\mf Q$. $A$ has a basis consisting $\{ 1, \omega, \Omega, \omega \Omega \}$ over $\mf Q$ and $\omega^2 = a, \Omega^2 =b, \omega\Omega + \Omega\omega = 0$. Here $a,b$ are square-free and we assume that $a>0$. For details on quaternion algebras, we refer the reader to \cite{ei}. 

Let $\alpha \in A$. We define, as usual, the trace and norm maps by $T(\alpha) = \alpha + \overline{\alpha}$ and $N(\alpha) =\alpha  \overline{\alpha} $. Here, $\overline{\alpha}$ is the conjugate to $\alpha$ defined by $\overline{\alpha} = x_0 - x_1 \omega - x_2\Omega - x_3 \omega \Omega$, if $\alpha =  x_0 + x_1 \omega + x_2\Omega + x_3 \omega \Omega$.

Recall that an order $S$ in $A$ is a subring of $A$ containing $1$, finite over $\mf Z$ and such that $S$ has $\mf Q$ basis of order $4$. Any such order is contained in a maximal order of $A$. Let $R$ be a maximal order of $A$ and $R(1)$ be its groups of units, i.e., elements of norm $1$. Further
let \[ R(n) = \{ \alpha \in R \mid N(\alpha) = n\}.  \]
It is well-known that the $\mf Q (\sqrt{a})$ algebra $A \otimes_{\mf Q} \mf Q (\sqrt{a})$ is split and so there exists an embedding $\phi$ of $A$ into $M_2( \mf Q(\sqrt{a}))$ defined by \[  \phi(\alpha) = \begin{bmatrix} \overline{\xi} & \eta \\ b \overline{\eta} & \xi  \end{bmatrix},     \]
where $ \alpha = x_0 + x_1 \omega + (x_2  + x_3 \omega) \Omega = \zeta + \eta \Omega$. Further it is known that $\det \phi(\alpha) = N(\alpha)$. In the sequel we will work with the image of $\phi$ in $M_2( \mf Q(\sqrt{a}))$, so we drop the $\phi$ from the notation for convenience. Since $A$ is a division algebra, $\Gamma \backslash \mf H$ is a compact hyperbolic surface. Thus any fundamental domain $\mc F$ for the action of $\Gamma$ on $\mf H$ is compact.

\subsection{Hecke operators}
From the theory of correspondences (see \cite{ei}), one can define Hecke operators $T(n)$, $n \geq 1$ using the orbits $R(1) \backslash R(n)$, which are $O(n^{1 + \varepsilon})$ in number, as follows. For $f \colon \mf H \rightarrow \mf C $ holomorphic one defines
\begin{align}
 f \mid T(n) := n^{k/2-1} \underset{\gamma \in R(1) \backslash R(n)}\sum f \mid_k \gamma,
\end{align}
where as usual, we denote 
\[ f \mid_k \gamma := (\det \gamma)^{k/2} (cz+d)^{-k} f \left( \frac{az+b}{cz+d} \right), \q \gamma =  \left( \begin{smallmatrix} a &b \\ c & d \end{smallmatrix} \right) \in \Gamma.\]

Let us denote the space of modular forms of weight $k$ for $\Gamma$ by $M_k(\Gamma)$. Analogous to the theory of modular forms for congruence subgroups of the modular group, one knows (see \cite{ei}) that there exist an integer $q$ depending on $R$\footnote{{In fact one has $q= q_1 q_2$, where $q_1$ is the product of the \lq characteristic primes\rq \ $p$ such that the} {local order $R_p := R \otimes_{\mf Z} \mf Z_p$ at the prime $p$} {is maximal, and $q_2$ is the product of those primes $p$ for which} {$R_p \cong \left( \begin{smallmatrix} \mf Z_p & \mf Z_p \\ p \mf Z_p & \mf Z_p   \end{smallmatrix} \right) $, see \cite[p.~38]{ei}}.} such that {the set of operators $T(n)$ with $(n,q)=1$}, preserve $M_k(\Gamma)$, are self-adjoint, and satisfy
\begin{align}
T(m)T(n) = \underset{d \mid (m,n)}\sum d^{k-1} T(\frac{mn}{d^2}).
\end{align}

\subsection{The Bergman kernel}
The Bergman kernel or the reproducing kernel for $\Gamma \backslash \mf H$ is characterized as the unique function (upto non-zero scalars) $\mc B(z,w)$ of two variables $z,w \in \mf H$ (holomorphic in $z$ and anti-holomorphic in $w$) such that for any holomorphic function $f$ on $\mf H$, one has (see \cite{kl}), writing $w = u + i v$:
\begin{align} \label{berg}
  \int_{\Gamma \backslash \mf H} v^k f(w) \mc B(z,w) \frac{d u d v}{v^2} = f(z).
\end{align}

The Bergman kernel for $M_k(\Gamma)$ can now be written down explicitly as follows. For $n \geq 1$, define the following function: 
\begin{align}
h^n_k(z,w) = \underset{\gamma = \left( \begin{smallmatrix} a & b \\ c & d  \end{smallmatrix} \right) \in R(n) }\sum n^{k/2} (c \overline{w} + d)^{-k} \left(\frac{ z - \gamma \overline{w}}{2i} \right)^{-k}.
\end{align}

It is easily checked that $h^n_k(z,w)$ defines a holomorphic function in $z$ and is anti-holomorphic in $w$. In this case it is well-known (or one can check directly from \eqref{berg}, see also \cite{cl, kl}) that 
\begin{align} \label{bergman}
\mc B(z,w) = \sum_{j=1}^{d} f_{j,k}(z) \overline{f_{j,k}(w)} =  2^{-1} (k-1) h_k(z,w),
\end{align}
where $f_{j,k}$ is any orthonormal basis of $M_k(\Gamma)$ (which is finite-dimensional) and thus $h_k$ is proportional to $\mc B$. In our paper we shall take the orthonormal basis to be the one consisting of $L^2$ normalized Hecke eigenforms.

\subsection{The convexity bound}\label{convex}
The `convexity bound' can be obtained as an application of Godement's theorem (see \cite{cl}, \cite{kl}) and a calculation in \cite{cl}. We record it here for the convenience of the reader.

Note that $ \mc B(z,z) =  \frac{k-1}{2} \cdot  h^1_k(z,z)$ and writing as in \eqref{IIdo} adopting the notation and estimates introduced in section~\ref{ber}
\begin{align} \label{conv}
\|f \|_\infty^2 \leq \frac{k-1}{2} \cdot \underset{\gamma \in \Gamma}\max \ |h_\gamma(z)|^{k -4} \cdot  \underset{\gamma \in \Gamma} \sum |h_\gamma(z)|^{4} \ll_{\Gamma} k.
\end{align}
and thus $\|f \|_\infty \ll_{\Gamma} k^{1/2}$.

\begin{rmk} \label{nohecke}
Note that \eqref{conv} holds without assuming that f is a Hecke cusp form, and indeed if one drops the Hecke assumption then this bound $k^{1/2}$ is sharp for some $f$. This follows from the Sarnak's multiplicity argument in his letter to Morawetz (see \cite{sarnakletter}), which shows that for some $f_0 \in M_k(\Gamma)$ one has \[ \| f_0 \|_{\infty}^2 \cdot \mrm{vol}(\Gamma \backslash \mf H) \geq \dim M_k(\Gamma) \approx k. \] 
\end{rmk}

\section{Estimation of the Bergman kernel} \label{ber}
In this section we carry out estimates for the Hecke-transformed Bergman kernel $h^n_k(z,w)$ in terms of $n$ and the imaginary parts of $z,w$. First we estimate it crudely, using the estimate for $h_k(z,w)$ as in \cite{kl}, and then use this in conjunction with a trick due to Cogdell-Luo in \cite{cl} to arrive at a reasonable estimate for $h^n_k(z,w)$.

We recall \textit{Godement's theorem} on the estimate for the majorant of $h_k(z,w)$ obtained by putting absolute values on its summands. Namely from \cite[p.~79, Prop.~2 (iii)]{kl} we obtain that (keeping in mind that Klingen's argument holds for any discrete subgroup of $SL(2, \mf R)$, see \cite[p.~81]{kl})
\begin{align} \label{h1}
\lceil h_k(z,w) \rceil \leq \alpha(K) (k-1)^{-1} \mrm{Im}(w)^{-k/2},
\end{align}

where $\lceil h \rceil$ signifies that the absolute values of the summands appearing in the definition of $h$ are considered. Here $w \in \mf H$, $z \in K$, and for any compact set $K \subset \mf H$ and $\alpha(K)$ is a constant depending only on $K$. We now note the following expression of $h^n_k(z,w)$ in terms of Hecke operators: 
\begin{equation} \label{tn}
n^{k/2 -1} h^n_k(z,w) =  h_k(z,w) \mid_{(w)} T(n),
\end{equation}
i.e.,
\[ n^{k/2 -1} h^n_k(z,w) = n^{k/2-1} \underset{\gamma \in R(1) \backslash R(n)}\sum h_k(z,w)\mid_{(w),k} \gamma \]
where the subscript $(w)$ denotes the variable on which the action is considered. Then using \eqref{h1}, we easily arrive at the following estimate:

\begin{align}\label{hn}
\lceil h^n_k(z,w)\rceil & \leq  \alpha(\mc F) (k-1)^{-1} n^{k/2} |j(\gamma,w)|^{-k} \underset{\gamma \in R(1) \backslash R(n)}\sum \mrm{Im}(\gamma w)^{-k/2} \nonumber  \\
& \leq  \alpha(\mc F) (k-1)^{-1} n^{1+ \varepsilon} \ \mrm{Im}(w)^{-k/2};
\end{align}
where we have taken the compact set $K$ to be the fundamental domain $\mc F$.

After this preliminary estimate, we now turn to a more refined estimate for $h^n_k(z,w)$. To this end, define, following \cite{cl}:
\[   h_\gamma (z) = \frac{y} { (z - \gamma \bar{z})/2i \cdot (c \bar{z} +d)}, \q \gamma = \begin{bmatrix}a & b \\ c & d \end{bmatrix}  \in R(n), y = \mrm{Im}(z).\]
Then 

\begin{align} \label{I,II}
y^k h^n_k(z,z) = n^{k/2} \underset{\gamma \in R(n)}\sum h_\gamma(z)^k = n^{k/2} ( \underset{\gamma \colon |\gamma z -z|\leq \delta}\sum + \underset{\gamma \colon |\gamma z -z| > \delta}\sum),
\end{align}
where $0 < \delta <1$ will be chosen later and we call the first and second terms I and II respectively. In I, we use the estimate
\begin{align}
|h_\gamma (z)| \leq \frac{2 y}{(y + \frac{n y}{|cz+d|^2} ) |cz+d|} \leq \frac{1}{n^{1/2}},
\end{align}
coupled with the following lemma:
\begin{lem} \label{zsmall}
For $0 < \delta <1$ small enough and $z \in \mc F$,
\[ \# \{ \gamma \in R(n) \colon | \gamma z - z| \leq \delta \} \leq n^\varepsilon (n \delta^{1/4} +1).  \]
\end{lem}

\begin{proof}
We will proceed as in \cite{is}. Namely, we consider the stabilizer of $z$ in $SL(2, \mf R)$ and call it $K_z$. It is a maximal compact subgroup of $\mf H$ and thus conjugate to $SO(2, \mf R)$ by a matrix, say, $M = \left( \begin{smallmatrix} a & b \\ c & d  \end{smallmatrix} \right)$. We next recall the Iwasawa decomposition in $SL(2, \mf R)$ with respect to $K_z$:
\begin{align} \label{iwasawa}
SL(2, \mf R) = N_z A_z K_z; \q \gamma = na k. 
\end{align}
First we would assume that $z=i$, and work with the standard Iwasawa decomposition with respect to the standard maximal compact subgroup $K=SO(2, \mf R)$. Here we have a canonical expression for $N,A$: 
\[ N = \left( \begin{matrix} 1 & \alpha' \\ 0 & 1  \end{matrix} \right), \q A = \left( \begin{matrix} \beta' & 0 \\ 0 & \beta'^{-1}  \end{matrix} \right),  \ (\alpha' \in \mf R, \beta' \in \mf R^{\times}).\]
Let $\gamma' \in SL(2, \mf R)$. Clearly, with the Iwasawa decomposition of $\gamma'$ and $\alpha', \beta'$ as above, one has $|\gamma' i - i| = | \alpha' + (\beta'^2  -1)i | = (\alpha'^2 + (\beta'^2  -1)^2)^{1/2}$. Thus $|\gamma' i - i| \leq \eta$ implies that 
\[ |\alpha'| \leq \eta, \q (1 - \eta)^{1/2} \leq |\beta'| \leq (1 + \eta)^{1/2}. \]
From these we also get $| \beta'|^{-1}  \leq 1 + c_1 \eta^{1/2}$ for some absolute constant $c_1>1$ and $\eta$ small enough. Thus the above inequalities show that for $\eta$ small enough, $\| p_i - I  \| \ll \eta^{1/2}$; where $p_i = n_ia_i$ from the decomposition \eqref{iwasawa} with respect to $z=i$. This implies, after multiplying by $SO(2, \mf R)$ that
\begin{align} \label{k}
\gamma' = k + O(\eta^{1/2}).
\end{align}
Now we can start from $\gamma \in R(n)$ such that $| \gamma z - z| \leq \delta $ and note that 
\[ K_z = \gamma_0 K \gamma_0^{-1}, \q \text{where } \gamma_0 = \left( \begin{smallmatrix} y^{1/2} & x y^{-1/2} \\ 0 & y^{-1/2} \end{smallmatrix}\right).  \]
Define $\gamma_1 := \gamma_0^{-1} \gamma \gamma_0$. From the inequalitiy:
\begin{align*}
|\gamma_0 Z - \gamma_0 W| \leq a \Rightarrow |Z-W| \leq a / y,
\end{align*}
for $a>0$, we find that
\begin{align*}
| (\gamma / n^{1/2}) z - z|  = | \gamma z - z|  \leq \delta \Rightarrow |\gamma_1 i - i| \leq \delta / y \leq c_2 \delta =: \eta,
\end{align*}
for some constant $c_2 >0$ depending only on $\Gamma$. Thus after conjugating \eqref{k} with $\gamma_0$ and $\eta$ defined as above:
\begin{align} \label{kz}
\gamma / n^{1/2} = k_z + O(\delta^{1/2}),
\end{align}
since the entries of $\gamma_0$ are bounded by some constant depending only on $\Gamma$. Now we can proceed as in \cite{is} by following the description of $K_z$ given there. We start with the quadratic form associated to $z$: 
\[ \alpha z^2 + \beta z + \gamma = 0, \q \alpha, \beta, \gamma \ll 1 \]
where $\alpha, \beta, \gamma$ are real. We allow ourselves to use the notation $\gamma$ both for a matrix and a real number in order to be consistent with the notation in \cite{is}, but will remind the reader in case of any possibility of confusion. The $\alpha, \beta, \gamma$ satisfy \footnote{One knows that $\alpha= y^{-1}/2, \beta = - x y^{-1}, \gamma = x^2 y^{-1}/2 + y/2$, in fact the quadratic form for $z$ is obtained by acting the matrix $\gamma_0^{-1}$ on the polynomial $X^2 + Y^2$.} 
\begin{align} \label{abg}
 \beta^{2} - 4 \alpha \gamma = -1.  
\end{align} 
From the explicit description of $K_z$ (see \cite[eq.~1.12]{is}), we find that 
\begin{align} \label{Kz}
K_z = \left \{ \begin{bmatrix} (t- \beta u)/2 & - \gamma u \\ \alpha u & (t + \beta u)/2  \end{bmatrix} \mid t^2 + u^2 = 4. \right \}
\end{align}
Thus, from the canonical description of $\gamma \in R(n)$ (see \cite[eq.~1.14]{is}):
\begin{align} \label{gamma}
\gamma  = \begin{bmatrix}  x_0 - x_1 \sqrt{a} & x_2 + x_3 \sqrt{a} \\ b x_2 -  b x_3 \sqrt{a} & x_0 + x_1 \sqrt{a} \end{bmatrix}
\end{align}
we get by comparing both sides of \eqref{kz}, using the descriptions in \eqref{Kz} and \eqref{gamma} that \footnote{Note that there is a typo in these equations in \cite{is}.} 
\begin{align} 
2 x_0  / n^{1/2} &= t + O(\delta^{1/2}). \label{x0} \\
2 x_1 \sqrt{a}  / n^{1/2} &= \beta u + O(\delta^{1/2}) \label{x1}  \\
2 x_2   / n^{1/2} &= - (\gamma - \alpha/b)u + O(\delta^{1/2}) \label{x2}\\
2 x_3 \sqrt{a} / n^{1/2} &= - (\gamma + \alpha/b)u + O(\delta^{1/2}). \label{x3} 
\end{align}  
Also, taking \eqref{abg} into account one obtains that
\begin{align*}
(\gamma + \alpha / b)^2 = (\gamma - \alpha / b)^2 + (1 + \beta^2)/b,
\end{align*}
which shows that either $|\gamma + \alpha / b| \geq 1/|b|$ or $|\gamma - \alpha / b| \geq 1/|b|$. Thus one of these quantities is bounded below uniformly for all $z \in \mc{F}$ (depending on the sign of $b$ only). 

First suppose that $b>0$. Then we have $|\gamma - \alpha / b| \geq 1/|b|$. The proof now follows that in \cite{is} and we obtain that
\begin{align}
4 = t^2 + u^2 = \frac{4 x^2_0}{n}  + \frac{4 a x_3^2}{n (\gamma + \alpha/b)^2} + O(\delta^{1/2}).
\end{align}
Taking into account \cite[Lemma~1.4]{is} we find
\begin{align} \label{x0x3}
\# \{x_0, x_3 \colon | {x^2_0}  + \frac{ a x_3^2}{ (\gamma + \alpha/b)^2} - n | \ll n \delta^{1/2} \} \leq n^\varepsilon (n \delta^{1/4} +1)
\end{align}
We have the standard estimate
\begin{align} \label{rs}
 \# \{ r,s \colon qr^2 + p s^2 = m; \ q \geq 1, p \geq 0  \} \ll m^\varepsilon,
\end{align}  
see the proof of \cite[Lemma~1.4]{is} for example. 

Recall that for fixed $x_0,x_3$, both of which are $\ll n^{1/2}$, (this follows directly from \eqref{kz}, \eqref{x0}, \eqref{x1}, \eqref{x2} and \eqref{x3}) the number of integral solutions $x_1,x_2$ (both of which are $\ll n^{1/2}$ by the same reason as above) is, by \eqref{rs} 
 \begin{align} \label{x1x2}
 \# \{ x_1, x_2 \ll n^{1/2} \colon ax_1^2 + b x_2^2 = x_0^2 + ab x_3^2 - n \} \ll |x_0^2 + ab x_3^2 - n|^\varepsilon \ll n^\varepsilon.
\end{align}
Thus combining \eqref{x0x3} and \eqref{x1x2} we see finally
\begin{align*}
& \# \{ \gamma \in R(n) \colon | \gamma z - z| \leq \delta \}  \ll n^\varepsilon (n \delta^{1/4} +1). \nonumber
\end{align*}
This settles the case $b>0$. When $b<0$ our choice would be $|\gamma - \alpha / b| \geq 1/|b|$, and this case is completely similar to the previous one. This completes the proof of the lemma.
\end{proof}

We are now in a position to derive a reasonable estimate for $h^n_k(z,z)$. Let us go back to \eqref{I,II}. For the sum I, we get the following estimate:
\begin{align} \label{Ido}
\underset{\gamma \colon |\gamma z -z|\leq \delta}\sum |h_\gamma(z)|^k \leq n^{-k/2} \cdot n^\varepsilon (n \delta^{1/4} +1).
\end{align} 
For the sum II, we do the following:
\begin{align}\label{IIdo}
\underset{\gamma \colon |\gamma z -z| > \delta}\sum |h_\gamma(z)|^k \leq (\underset{\gamma \colon |\gamma z -z| > \delta}\max |h_\gamma(z)|^{k -k_0}) \cdot  \underset{\gamma} \sum |h_\gamma(z)|^{k_0},
\end{align}
where $k_0 >2$ is a positive integer to be chosen later and use the estimate \eqref{hn} for $\underset{\gamma} \sum |h_\gamma(z)|^{k_0} = \lceil h^n_{k_0}(z,z) \rceil$. Next, \cite[Lemma~1]{cl} shows that
 \begin{align} \label{del}
 |h_\gamma(z)| \leq (1 + \delta^2)^{- 1/2}, \q \text{if} \q |\gamma z -z| > \delta.
\end{align}
We remind the reader that it is easy to see that \cite[Lemma~1]{cl} holds for all $\gamma \in GL_2^+(\mf R)$ with $\det \gamma \geq 1$. We put \eqref{del} in \eqref{IIdo}. Thus from \eqref{I,II}, using \eqref{Ido} and \eqref{IIdo} together we have,
\begin{align*} 
y^k |h^n_k(z,z)| \ll  n^\varepsilon (n \delta^{1/4} +1) +  O \left(  n^{1 + \varepsilon}  (1 + \delta^2)^{- {(k - k_0)}/2} \right).
\end{align*}
We are now in a position to prove \thmref{mainthm}. 

\section{Proof of \thmref{mainthm}}
In this section we will prove \thmref{mainthm}. First we choose a value for $\delta$ which gives rise to a decay in terms of $k$ in the sum II in the $h^n_k(z,z)$. We also note that $\delta$ will depend on $n$ (used in estimating $h^n_k(z,z)$) but we suppress it in notation for convenience. We use the results of the previous section along with the amplification technique of \cite{is} to get our result.

\begin{proof} To begin with, let us choose 
\begin{align}
\delta := \frac{C}{n^\beta};
\end{align}
where $C$ is a sufficiently small positive constant depending only on the group $\Gamma$ such that \lemref{zsmall} holds and $ \beta > 0$ would be chosen later. The estimate for $h^n_k(z,z)$ now reads:
\begin{align} \label{hnfinal}
y^k |h^n_k(z,z)| \ll  n^{1 - \beta/4 + \varepsilon} +  O \left(  n^{1 + \varepsilon}  (1 + Cn^{- 2\beta})^{- {(k-k_0)}/2} \right).
\end{align}
We define the `normalized' eigenvalues for each $1 \leq j \leq d$,
\begin{align} \label{etalam}
\eta_j(n) := \lambda_j(n)/ n^{(k-1)/2};
\end{align}
and then the Hecke relation takes the form
\begin{align}
\eta_j(m) \eta_j(n) = \underset{d \mid (m,n)}\sum \eta_j(mn/d^2). 
\end{align}
We start from the equalities
\begin{align}
\sum_{j=1}^d y^k f_j(z)  \overline{f_j(w)}  \underset{n \leq N}\sum |\alpha_n \eta_j(n)^2  | &=
y^k \underset{m,n}\sum \alpha_n \overline{\alpha_m} \underset{j}\sum f_j(z)  \overline{f_j(w)}  \eta_j(m) \eta_j(n) \nonumber  \\
& = \underset{m,n}\sum \alpha_n \overline{\alpha_m} \underset{d \mid (m,n)}\sum \underset{j} \sum y^k f_j(z)   \eta_j(mn/d^2) \overline{f_j(w)}  \label{norma} \\
&= \underset{m,n}\sum \alpha_n \overline{\alpha_m} \underset{d \mid (m,n)}\sum \frac{d}{(mn)^{1/2} } y^k h^{\frac{mn}{d^2}}_k(z,w) \label{ens}
\end{align}
where to go from \eqref{norma} to \eqref{ens} we have used the relation \eqref{tn} between $h^{n}_k(z,w)$ and $h_k(z,w)$ with respect to the Hecke operator $T(n)$, and the normalisation \eqref{norma}.

Let us fix a positive integer $M$ (which would be absolutely bounded) and choose the integer $k_0$ appearing in \eqref{IIdo} by $k_0=4$. For convenience, let  us define $\kappa =(k-4)/2$. Now for $k>2M+4$ (so that $\kappa>M$), we have the inequality
\begin{equation} \label{binom}
(1 + Cn^{- 2\beta})^{- \kappa} \ll \frac{n^{2 M \beta}}{\binom{\kappa}{M}}, 
\end{equation} 
in \eqref{hnfinal} and 
obtained by retaining only the $M$-th term in the above binomial expansion. Here the implied constant depends only on $\Gamma$. Now we use \eqref{binom} in the estimate of $h^n_k(z,z)$ as in \eqref{hnfinal} and obtain from \eqref{ens} the following:
\begin{align}
& \sum_{j=1}^d y^k f_j(z)  \overline{f_j(w)}  \underset{n \leq N}\sum |\alpha_n \eta_j(n)^2  | \nonumber \\
& \ll k \left( \underset{m,n \leq N}\sum |\alpha_n| |\alpha_m| \underset{d \mid (m,n)}\sum \frac{d}{(mn)^{1/2} } \left( (\frac{mn}{d^2})^{1 - \beta/4 + \varepsilon} + \frac{(\frac{mn}{d^2})^{1 + 2 M \beta + \varepsilon}}{\binom{k}{M}}) \right)  \right).  \label{anil}
\end{align}
Let us choose $\beta = 4 +4\varepsilon$ and split the sum in \eqref{anil} into two parts $I$ and $II$ corresponding to the two summands inside the second summation. The first sum contributes
\begin{align}
\ll k  \underset{m,n \leq N}\sum |\alpha_n| |\alpha_m| \underset{d \mid (m,n)}\sum \frac{d}{(mn)^{1/2} }  
 &= \underset{dm_1,dn_1 \leq N}\sum \frac{ |\alpha_{dn_1}| |\alpha_{dm_1}| } { (m_1n_1)^{1/2} } \label{anil1}\\
& \ll \underset{dm_1,dn_1 \leq N}\sum \left( \frac{ |\alpha_{dn_1}|^2}{m_1}  + \frac{|\alpha_{dm_1}|^2 } { n_1 } \right) \label{amgm}
\end{align}
where the sum in \eqref{anil1} is over $d,m_1,n_1$ and 
have used the AM-GM inequality on $|\alpha_m|$ and $|\alpha_n|$ to arrive at \eqref{amgm}. Clearly,
\begin{align}
\underset{dm_1,dn_1 \leq N}\sum \frac{ |\alpha_{dn_1}|^2}{m_1} \ll \underset{dn_1 \leq N}\sum |\alpha_{dn_1}|^2 \log N \ll N^\epsilon \underset{n_2 \leq N}\sum |\alpha_{n_2}|^2 \sum_{d \mid n_2} 1 \ll N^\epsilon \underset{n_2 \leq N}\sum |\alpha_{n_2}|^2,
\end{align}
and similarly for the other sum in \eqref{amgm}. Therefore the contribution from $I$ is at most
\begin{align} \label{ok1}
kN^\epsilon \underset{n_ \leq N}\sum |\alpha_{n}|^2.
\end{align}
The part $II$ contributes at most
\begin{align}
k  \underset{m,n \leq N}\sum |\alpha_n| |\alpha_m| \underset{d \mid (m,n)}\sum \frac{(\frac{mn}{d^2})^{1/2 + 8 M + \varepsilon}}{\binom{\kappa}{M}}) 
& \ll \frac{ kN^{1 + 16 M + \epsilon} } { \binom{\kappa}{M} } (\underset{n_ \leq N}\sum |\alpha_{n}|)^2. \label{ok2}
\end{align}
Therefore from \eqref{amgm}, \eqref{ok1} and \eqref{ok2} we get
\begin{align}
\sum_{j=1}^d y^k f_j(z)  \overline{f_j(w)}  \underset{n \leq N}\sum |\alpha_n \eta_j(n)^2  | 
& \ll kN^\epsilon \underset{n_ \leq N}\sum |\alpha_{n}|^2 + \frac{ kN^{1 + 16 M + \epsilon} } { \binom{\kappa}{M} } \left(\underset{n_ \leq N}\sum |\alpha_{n}| \right)^2. \label{bla}
\end{align}
We now would use the amplification method to arrive at an estimate of the sup-norm as follows. Let us fix an eigenform $f_{j_0}$. The choice for $\alpha_n$ is the same as in \cite{is}, namely
\begin{align} \label{alpha}
\alpha_n = \begin{cases} \eta_{j_0}(p) & \text{ if } n = p \leq N^{1/2} \\
- 1 & \text{ if } n = p^2 \leq N \\ 
0 & \text{ otherwise }. \end{cases}
\end{align}

Recall that under the Jacquet-Langlands correspondence \cite[p. check]{jl} and also \cite{he2}, there exists a cusp form $F_{j_0}$ of weight $k$ on $\Gamma_0(D)$ with $D$ depending only on the order $R$ such that the Hecke eigenvalues of $F_{j_0}$ coincide with those of $f_{j_0}$ for all $(n,q)=1$. Thus Deligne's bound holds for $\eta_{j_0}(n)$ for $(n,q)=1$.

Keeping in mind the Hecke relation for primes $p \nmid q$:
\[ \eta_{j_0}^2(p) -  \eta_{j_0}(p^2) = 1,\]
and using Deligne's bound and \eqref{alpha}, we conclude from \eqref{bla}, that
\begin{align*} 
\| f_{j_0} \|_\infty^2 \left(  \underset{p \leq N^{1/2}, p \nmid q}\sum 1 \right)^2 &\ll k N^{1/2 + \varepsilon} + \frac{k N^{1 + 16 M + \varepsilon}}{ \binom{\kappa }{M}} \\
&\ll k N^{1/2 + \varepsilon} + \frac{N^{2 + 16 M + \varepsilon}}{ \kappa^{M-1}}  .
\end{align*}
Here we have assumed that $M$ to be absolutely bounded. Note that the implied constants depend only on $\epsilon$ and $\Gamma$. We obtain finally
\begin{align*}
\| f_{j_0} \|_\infty^2  \ll k N^{- 1/2 + \varepsilon} + \frac{N^{1 + 16 M + \varepsilon}}{k^{M-1}}.
\end{align*}
We choose $N$ by \[ N = k^{\frac{M}{3/2 + 16 M +  \varepsilon}} \] to obtain for large $k$:
\[ \| f_{j_0} \|_\infty^2 \ll k^{1 - \frac{M/2}{3/2 + 16 M +\varepsilon} + \varepsilon} \q \text{or,} \q \| f_{j_0} \|_\infty \ll k^{\frac{1}{2} - \frac{M/4}{3/2 + 16M } + \varepsilon} . \] 
Finally we choose $M=4$ in the above and this completes the proof of \thmref{mainthm}.
\end{proof}

\begin{rmk}
Clearly, the bound improves as $M$ increases. However the rate of improvement is negligible. For example, $M=4$ produces the exponent $0.4847..$ in \thmref{mainthm}, whereas $M=100$ produces the exponent $0.4843..$, the limit being $1/64$. 
\end{rmk}


\begin{thebibliography}{22}


\bibitem{cl} J. W. Cogdell, W. Luo, {\em The Bergman kernel and mass equidistribution on the Siegel modular variety $\mrm{Sp}_{2n}(\mbb Z)\backslash \mathfrak{H}_n$}. Forum Math. \x{23} (2011), no. 1, 141--159.

\bibitem{ei} M. Eichler, {\em Lectures on modular correspondences}. Tata Institute Lecture notes, \x{9}, 1955.

\bibitem{go} R. Godement, {\em S\'{e}rie de Poincar\'e et Spitzenfprmen}. Fonctiones automorphes, vol. \x{1}, Expos\'e \x{10}, S\'{e}minaire H. Cartan, Paris, 1957/1958.

\bibitem{he1} D. Hejhal, {\em the Selberg Trace Formula for $PSL_2(\mbb R)$}. vol. \x{548}, Lecture notes in Mathematics, Springer-Verlag, 1976.

\bibitem{he2} D. Hejhal, {\em A classical approach to a well known correspondence on quaternion groups}. vol. \x{1135}, Lecture notes in Mathematics, Springer-Verlag, 1985.

\bibitem{he3} D. Hejhal, {\em the Selberg Trace Formula for $PSL_2(\mbb R)$}. vol. \x{1001}, Lecture notes in Mathematics, Springer-Verlag, 1980.

\bibitem{is} H. Iwaniec, P. Sarnak, {\em $L^\infty$ norms of eigenfunctions of arithmetic surfaces}. Ann. of Math. (2) \x{141} (1995), no. 2, 301--320.

\bibitem{jl} H. Jacquet, R. Langlands, {\em Automorphic forms on $GL(2)$}. Vol. \x{114}, Lecture notes in Mathematics, Springer-Verlag, 1970.

\bibitem{kl} H. Klingen, {\em Introductory lectures on Siegel modular forms}. Cambridge Studies in Advanced Mathematics, \x{20}, Cambridge University Press, 1990.

\bibitem{sarnakletter} P. Sarnak, {\em Letter to C. Morawetz}. (2004).

\bibitem{xia} H. Xia, {\em On $L^\infty$ norms of holomorphic cusp forms}. J. Number Theory \x{124}, (2007), 325--327.
\end{thebibliography}
\end{document}